\documentclass[11pt, reqno]{amsart}
\usepackage[margin=1in]{geometry}
\usepackage{mathtools, amssymb, amsthm}
\usepackage{cleveref}
\usepackage[backend=biber, style=alphabetic]{biblatex}
\usepackage{tikz}
\usetikzlibrary{cd}
\newcommand{\mc}{\mathcal}      % shorthand for \mathcal
\newcommand{\set}[2]{\left\{#1 : #2\right\}} % set notation
\newcommand{\powerset}[1]{{\mathrm{pow}(#1)}} % powerset
\newcommand{\flats}[1]{L(#1)}  % set of flats
\newcommand{\closure}[1]{\operatorname{cl}_{#1}}
\newcommand{\circuits}[1]{\operatorname{\mathcal{C}}(#1)}
\newcommand{\N}{\mathbb{N}}
\DeclareMathOperator{\id}{id} % identity map
\DeclareMathOperator{\Hom}{Hom} % hom functor
\DeclareMathOperator{\colim}{colim} % colimit
\newcommand{\MS}{\mathbf{Mat}_{\bullet}}   % pointed matroids with strong maps
\newcommand{\MSF}{\mathbf{Mat}_{\bullet}^{\textrm{fin}}}   % pointed matroids with strong maps
\newcommand{\SET}{\mathbf{Set}_{\bullet}} % pointed sets with pointed maps
\newcommand{\fin}[1]{{#1}^{\operatorname{fin}}} % finitization functor
\newcommand{\E}{\mc{E}}                 % proto-exact category
\newcommand{\EE}{\mathfrak{E}}          % admissible epics
\newcommand{\MM}{\mathfrak{M}}          % admissible monics
\newcommand{\FF}{\mathbb{F}}            % forgetful functor
\newcommand{\GG}{\mathbb{G}}            % functor
\theoremstyle{definition}
\newtheorem{mydef}{Definition}[section]
\newtheorem{myeg}[mydef]{Example}

\newtheorem{rmk}[mydef]{Remark}
\newtheorem{que}[mydef]{Problem}
\theoremstyle{plain}
\newtheorem{mythm}[mydef]{Theorem}
\newtheorem{lem}[mydef]{Lemma}
\newtheorem{prop}[mydef]{Proposition}
\newtheorem{cor}[mydef]{Corollary}
\title{On Infinite Matroids with Strong Maps\\ Proto-Exactness and Finiteness Conditions}
\author{Chris Eppolito}
\address{Department of Mathematical Sciences, Binghamton University, Binghamton, NY, 13902, USA}
\email{eppolito-at-math-dot-binghamton-dot-edu}
\author{Jaiung Jun}
\address{Department of Mathematics, State University of New York at New Paltz, NY 12561, USA}
\email{junj@newpaltz.edu}

\makeatletter
\@namedef{subjclassname@2020}{%
  \textup{2020} Mathematics Subject Classification}
\makeatother

\begin{document}
\begin{abstract}
  This paper investigates infinite matroids from a categorical perspective.
  We prove that the category of infinite matroids is a proto-exact category in the sense of Dyckerhoff and Kapranov, thereby generalizing our previous result on the category of finite matroids.
  We also characterize finitary matroids as co-limits of finite matroids, and show that the finite matroids are precisely the finitely presentable objects in this category.
\end{abstract}

\maketitle

%%%%%%%%%%%%%%%%%%%%%%%%%%%%%%%%%%%%%%%%%%%%%%%%%%%%%%%%%%%%%%%%%%%%%%%%
\section{Introduction}

Matroids combinatorially generalize properties of linear dependence in a finite-dimensional vector space.
Finite matroids naturally appear in various fields in mathematics.
Meanwhile, theories of infinite matroids have long sought to generalize finite matroids to infinite ground sets while retaining nice combinatorial properties as in the finite case, especially the multitude of equivalent definitions (i.e.\ cryptomorphisms) they enjoy.
Our work uses the definition of infinite matroids from \cite{bruhn2013axioms}, which successfully generalizes five of the most common axiomatizations from the finite case.

Recent years have seen growing interest in categorical aspects of matroids (e.g.\ see \cite{heunen2018category, eur2020logarithmic}).
In a previous paper, the authors (with M.~Szczesny) studied the category of finite matroids with strong maps in connection with combinatorial Hopf algebras, and initiated the study of algebraic K-theory for finite matroids.
Our main observation was that the category of finite matroids has the structure of a finitary proto-exact category in the sense of \cite{dyckerhoff2019higher}.

In this paper we define a category $\MS$ of infinite matroids which contains the category of finite matroids and strong maps as a full subcategory.
\Cref{theorem: main theorem} proves that $\MS$ has the structure of a proto-exact category, thereby generalizing \cite[Theorem 5.11]{eppolito2018proto}.
\Cref{proposition: finitary matroids} characterizes finitary matroids as the co-limits of finite matroids; along the way we obtain \Cref{cor: induce fin}, yielding a ``finitization'' functor $\MS \to \MSF$.
\Cref{proposition: finitely presentable} shows that the finitely presentable objects of $\MS$ are precisely the finite matroids.
We end with a short discussion of some open problems.

%%%%%%%%%%%%%%%%%%%%%%%%%%%%%%%%%%%%%%%%%%%%%%%%%%%%%%%%%%%%%%%%%%%%%%%%
\section{Preliminaries}\label{section: prelim}

We denote the powerset of $E$ by $\powerset{E}$ and denote set of non-negative integers by $\N$.

\subsection{Infinite matroids}

First recall the definition of infinite matroid from \cite[Sections 1.3 and 1.4]{bruhn2013axioms}.
A collection $\mc{A} \subseteq \powerset{E}$ has \emph{Property $(M)$} when for all $A \in \mc{A}$ and all $S \subseteq E$ the subcollection $\mc{A}(A, S) = \set{X \in \mc{A}}{A \subseteq X \subseteq S}$ has a maximal element.

\begin{mydef}
  A function $\closure{} \colon \powerset{E} \to \powerset{E}$ is a \emph{matroid closure} on $E$ when the following hold:
  \begin{enumerate}
  \item[(CLO)] $\closure{}$ is a closure operator, i.e.,~extensive, monotone, and idempotent on $\powerset{E}$.
  \item[(CLE)] For all $Z \subseteq E$ and all $x, y \in E$, if $y \in \closure{}(Z \cup x) \setminus \closure{}(Z)$, then $x \in \closure{}(Z \cup y)$.
  \item[(CLM)] The collection $\mc{I}_{\closure{}} \coloneqq \set{ I \subseteq E }{x \notin \closure{}(I \setminus x) \text{ for all } x \in I}$ has Property $(M)$.
  \end{enumerate}
\end{mydef}

\begin{mydef}
  A collection $\mc{C} \subseteq \powerset{E}$ is the set of \emph{circuits} of a matroid on $E$ when the following hold:
  \begin{enumerate}
  \item[(C0)] $\emptyset \notin \mc{C}$.
  \item[(CI)] Elements of $\mc{C}$ are pairwise incomparable with respect to inclusion.
  \item[(CE)] For all $X \subseteq C \in \mc{C}$, all families $\set{ C_x \in \mc{C} }{ x \in X }$ satisfying $x \in C_y$ iff $x = y$, and for all $z \in C \setminus \bigcup_{x} C_x$, there is a $C' \in \mc{C}$ such that $z \in C' \subseteq \left( C \cup \bigcup_x C_x \right) \setminus X$.
  \item[(CM)] The collection $\mc{I}_{\mc{C}} \coloneqq \set{ I \subseteq E }{ C \not \subseteq I \text{ for all } C \in \mc{C}}$ has Property $(M)$.
  \end{enumerate}
\end{mydef}

These are ``cryptomorphic'' descriptions of a matroid on $E$ by \cite{bruhn2013axioms}.
These axiomatizations agree with standard definitions of finite matroids; see \cite[Appendix]{white1986theory}, noting that (CLM) and (CM) are redundant by finiteness of the ground set.

A \emph{loop} of $M$ is an element of $\closure{M}(\emptyset)$ (equivalently an element $e \in E$ with $\{e\}$ a circuit of $M$).
A \emph{pointed matroid} is a matroid with a distinguished loop $\ast$.
We often denote the ground set of a (pointed) matroid $M$ by $E(M)$, or simply $E$ when the matroid is clear from context.

\begin{myeg}
  The set $\mc{C}_r(E)$ of $r$-subsets of $E$ is the set of circuits of $U_r(E)$, the \emph{uniform matroid of rank $r$ on $E$}, for all positive integers $r$.
  Similarly, the set $\mc{C}^*_r(E)$ of subsets of $E$ with complement an $r$-subset of $E$ is the set of circuits of $U_r^*(E)$, the \emph{uniform matroid of corank $r$ on $E$}, for all positive integers $r$.
\end{myeg}

By \emph{matroid}, we mean \emph{possibly infinite pointed matroid} unless otherwise stated.
In particular, we explicitly state when ground sets are finite, and all of our matroids come with a distinguished loop.
We often suppress the word ``pointed'' in our terminology.

\begin{mydef}\label{definition: restriction and contraction}
  Let $M$ be a matroid with circuit set $\mc{C}$ and let $S \subseteq E$.
  The \emph{restriction} of $M$ to $S$ is the matroid $M | S$ on ground set $S$ with circuit set $\mc{C} | S = \set{C \in \mc{C}}{C \subseteq S}$.
  The \emph{contraction} of $M$ by $S$ is the matroid $M / S$ on ground set $E \setminus S$ with circuits the nonempty, inclusion-wise minimal elements of $\mc{C} / S = \set{C \setminus S}{C \in \mc{C}}$.
\end{mydef}

When working with pointed matroids, we consider only pointed restrictions and pointed contractions, i.e.\ restrictions and contractions which preserve the distinguished loop.
Effectively, this means pointed restrictions have the form $M | (S \cup \ast)$ and pointed contractions have the form $M / (S \setminus \ast)$.

%%%%%%%%%%%%%%%%%%%%%%%%%%%%%%%%%%%%%%%%%%%%%%%%%%%%%%%%%%%%%%%%%%%%%%%%
\subsection{Proto-exact categories}

\begin{mydef}\label{proto_exact}
  A {\em proto-exact} category is a pointed category $\mc{C}$ equipped with two distinguished classes of morphisms, \emph{admissible monomorphisms} $\MM$ and \emph{admissible epimorphisms} $\EE$, which satisfy the following conditions:
  \begin{enumerate}
  \item Every morphism $0 \rightarrow M$ is in $\MM$, and every morphism $M \rightarrow 0$ is in $\EE$.
  \item The classes $\MM$ and $\EE$ are closed under composition and contain all isomorphisms.
  \item A commutative square of the following form in $\E$ with $i,i' \in \MM$ and $j, j' \in \EE$
    \begin{equation}\label{eq: biCartesian}
      \begin{tikzcd}
        M\ar[r,"i"]\ar[d,"j",swap]& N
        \ar[d,"j'"]\\
        M'\ar[r,"i'"]& N'
      \end{tikzcd}
    \end{equation}
    is Cartesian if and only if it is co-Cartesian.
  \item Every diagram \(
    \begin{tikzcd}
      M'\ar["i'",r]&N'&N\ar["j'",l,swap]
    \end{tikzcd}
    \) with $i' \in \MM$ and $j' \in \EE$ can be completed to a bi-Cartesian square \eqref{eq: biCartesian} with $i \in \MM$ and $j \in \EE$.
  \item Every diagram \(
    \begin{tikzcd}
      M'&M\ar["j",l,swap]\ar["i",r]&N
    \end{tikzcd}
    \) with $i \in \MM$ and $j \in \EE$ can be completed to a bi-Cartesian square \eqref{eq: biCartesian} with $i' \in \MM$ and $j' \in \EE$.
  \end{enumerate}
\end{mydef}

See \cite{eppolito2018proto} for examples and motivation regarding proto-exact categories.
For proto-exact categories in connection with algebraic geometry, see \cite{eberhardt2020algebraic, jun2020toric}.

%%%%%%%%%%%%%%%%%%%%%%%%%%%%%%%%%%%%%%%%%%%%%%%%%%%%%%%%%%%%%%%%%%%%%%%%
\section{Proto Exactness of $\MS$ and $\MSF$}\label{section: main theorem}

In this section we prove that the category of pointed infinite matroids is proto-exact (\Cref{theorem: main theorem}).

Let $E$ be a set.
For $X \subseteq E$ and $e \in E$, let $X \cup e$ denote the union $X \cup \{e\}$.
The proof of the following lemma is straightforward, mimicking the proof of the finite case.

\begin{lem}\label{closure by circuits}
  Let $M$ be a matroid on $E$.  Then
  \[
    \closure{M}(X)
    =
    X \cup \set{e \in E}{\text{there is a circuit satisfying }e \in C \subseteq X \cup e}.
  \]
\end{lem}

We say $X \subseteq E$ is a \emph{flat} of $M$ when $\closure{M}(X) = X$.

\begin{lem}\label{lattice of flats}\label{join and meet}
  The set $\flats{M}$ of flats of a matroid $M$ under inclusion is a complete, atomic lattice.
  Meet and join are given by $X \wedge Y = X \cap Y$ and $X \vee Y = \closure{M}(X \cup Y)$ respectively.
\end{lem}

\begin{proof}
  Let $X, Y \in \flats{M}$.
  Notice $X, Y \subseteq \closure{M}(X \cup Y)$; if $X, Y \leq Z$, then $X \cup Y \subseteq Z$ yields $\closure{M}(X \cup Y) \leq \closure{M}(Z) = Z$.
  Hence $X \vee Y = \closure{M}(X \cup Y)$ and $\flats{M}$ is a join semilattice.
  If $Z \leq X, Y$, then $Z \subseteq X \cap Y \subseteq \closure{M}(X \cap Y)$; thus $X \wedge Y \subseteq \closure{M}(X \cap Y)$.
  Assuming to the contrary that $e \in \closure{M}(X \cap Y) \setminus (X \cap Y)$, there is a circuit satisfying $e \in C \subseteq X \cap Y \cup e$.
  Thus $e \in \closure{M}(X) \cap \closure{M}(Y) = X \cap Y$, contradicting our assumption.
  Completeness follows from the fact that $\flats{M}$ is the collection of closed sets of a closure operator.
\end{proof}

\begin{lem}\label{flats and operations}
  Let $M$ be a matroid.
  For all $S \subseteq E(M)$ we have both
  \begin{align*}
    \flats{M / S}
    & = \set{F \setminus S}{S \subseteq F \in \flats{M}}
    &
    &\text{and}
    &
      \flats{M | S}
    & = \set{F \cap S}{F \in \flats{M}}
      .
  \end{align*}
  The corresponding matroid closure operators are given by
  \begin{align*}
    \closure{M / S}(T)
    & = \closure{M}(T \cup S)\setminus S
    &
    &\text{and}
    &
      \closure{M | S}(T)
    & = \closure{M}(T) \cap S
      .
  \end{align*}
\end{lem}

\begin{proof}
  The full claim follows from our description of the closure operator.
  We denote the collection of minimal nonempty sets in a family $\mc{F}$ by $\min(\mc{F})$.

  We compute the closure on the contraction for all $T \subseteq E(M) \setminus S$ via
  \begin{align*}
    \closure{M / S}(T)
    & = T \cup \set{e \in E(M) \setminus S}{\exists C \in \circuits{M / S} \text{ with } e \in C \subseteq T \cup e}
    \\
    & = T \cup \set{e \in E(M) \setminus S}{\exists C \in \min(D \setminus S : D \in \circuits{M}) \text{ with } e \in C \subseteq T \cup e}
    \\
    & = T \cup \set{e \in E(M) \setminus S}{\exists C \in \circuits{M} \text{ with } e \in (C \setminus S) \subseteq T \cup e}
    \\
    & = T \cup \set{e \in E(M) \setminus S}{\exists C \in \circuits{M} \text{ with } e \in C \subseteq T \cup S \cup e}
    \\
    & = T \cup (\closure{M}(T \cup S) \setminus S)
    \\
    & = \closure{M}(T \cup S) \setminus S
      .
  \end{align*}

  We compute the closure on the restriction for all $T \subseteq S$ via
  \begin{align*}
    \closure{M | S}(T)
    & = T \cup \set{e \in S}{\exists C \in \circuits{M | S} \text{ with } e \in C \subseteq T \cup e}
    \\
    & = T \cup \set{e \in S}{\exists C \in \circuits{M} \text{ with } C \subseteq S \text{ and } e \in C \subseteq T \cup e}
    \\
    &  = T \cup \set{e \in S}{\exists C \in \circuits{M} \text{ with } e \in C \subseteq T \cup e}
    \\
    & = T \cup (\closure{M}(T) \cap S)
    \\
    & = \closure{M}(T) \cap S
      .
  \end{align*}
  The claimed formulas for $\flats{M | S}$ and $\flats{M / S}$ now follow trivially.
\end{proof}

We thus obtain the usual result that the deletion and contraction operations commute.
We state this result below in terms of restriction and contraction (the proof is routine).

\begin{cor}\label{restriction-contraction}
  If $M$ is a matroid with disjoint $S, T \subseteq E$, then $(M | (S \cup T)) / T = (M / T) | S$.
\end{cor}
Next we discuss the maps of our category.

\begin{prop}\label{strong map conditions}
  Let $M$ and $N$ be matroids and $f \colon E(M) \to E(N)$ be a map of ground sets.  The following are equivalent.
  \begin{enumerate}
  \item \label{closure containment} %
    For all $A \subseteq E(M)$ one has $f (\closure{M}(A)) \subseteq \closure{N}(f( A))$.
  \item \label{flat preimages} %
    The preimage of every flat of $N$ is a flat of $M$.
  \item \label{lattice morphism} %
    The map $f^{\#} \colon \flats{M} \to \flats{N}$ induced by $f$ is a morphism of complete lattices which restricts to a map $A_M \to A_N \cup \{0\}$ of atoms.
  \end{enumerate}
\end{prop}

\begin{proof}
  This follows from \cite[Proof of Proposition 8.1.3]{white1986theory}, noting completeness of the corresponding lattices for $\eqref{lattice morphism} \implies \eqref{closure containment}$.
\end{proof}

\begin{mydef}
  Let $M$ and $N$ be matroids.
  A \emph{strong map} of (pointed) matroids is a map $f \colon E(M) \to E(N)$ of (pointed) sets satisfying any (hence all) of the conditions given in \Cref{strong map conditions}.
\end{mydef}

Let $\MS$ be the category of pointed matroids with pointed strong maps.
\Cref{flats and operations}, \Cref{restriction-contraction}, and \Cref{canonical maps} all hold for pointed restrictions and pointed contractions.
The details from the finite case in \cite[Section 2]{eppolito2018proto} carry over to the infinite case verbatim.

We obtain the following from \Cref{strong map conditions} and \Cref{flats and operations}.

\begin{cor}\label{canonical maps}
  Restriction $M | S \xrightarrow{i_S} M$ and contraction $M \xrightarrow{c_S} M / S$ are strong maps.
\end{cor}

The map $c_S$ from \Cref{canonical maps} is well defined as a function because we consider \emph{pointed} matroids.
Indeed, this is our main motivation for using this convention.

The lemma below follows from \cite[Corollary 3.8]{heunen2018category}.

\begin{lem}\label{monic and epic}
  Let $\FF \colon \MS \to \SET$ be the forgetful functor sending a pointed matroid $(M,*_M)$ to its pointed ground set $(E(M),*_M)$.
  Let $M \xrightarrow{f} N$ be a pointed strong map.
  \begin{enumerate}
  \item $f$ is monic in $\MS$ if and only if $\FF(f)$ is monic in $\SET$.
  \item $f$ is epic in $\MS$ if and only if $\FF(f)$ is an epic in $\SET$.
  \end{enumerate}
\end{lem}

Let $\MM$ be the class of pointed strong maps which factor as $N \xrightarrow{~} M|S \xrightarrow{i_S} M$, and let $\EE$ consist of all pointed strong maps which factor as $M \xrightarrow{c_S} M/S \xrightarrow{~} N$.

\begin{mythm}\label{theorem: main theorem}
  The triple $(\MS, \MM, \EE)$ is a proto-exact category.
\end{mythm}

\begin{proof}
  we have proved the following extensions of our results in \cite[Section 5]{eppolito2018proto}:
  \begin{enumerate}
  \item \Cref{flats and operations}; For $T \subseteq S \subseteq E$, one has
    \[
      \flats{(M | S) / T} = \set{(F \cap S) \setminus T}{T \subseteq F \in \flats{M}}.
    \]
  \item \Cref{restriction-contraction}; For disjoint $S, T \subseteq E$, one has
    \[
      (M | (S \cup T)) / T = (M / T) | S.
    \]
  \item \Cref{canonical maps}; The restriction and contraction maps are strong maps.
  \item \Cref{monic and epic}; A strong map is monic (resp.~epic) in $\MS$ if and only if it is injective (resp.~surjective) on ground sets.
  \end{enumerate}
  The proof from \cite[Section 5]{eppolito2018proto} now shows that $\MS$ is a proto-exact category.
\end{proof}

%%%%%%%%%%%%%%%%%%%%%%%%%%%%%%%%%%%%%%%%%%%%%%%%%%%%%%%%%%%%%%%%%%%%%%%%
\section{Finiteness Conditions}

The propositions in this section illustrate the utility of applying categorical language to study infinite matroids.
We focus primarily on two finiteness conditions for matroids in this analysis.
The first condition (i.e.\ finitarity) is a well-studied condition in the literature on infinite matroids---indeed, this was one of the earliest definitions for infinite matroids.
The second condition (i.e.\ finite presentability) is a well-studied categorical notion, which has not been considered in the context of infinite matroids to the best of our knowledge.

\subsection{Finitary Matroids}

A matroid is \emph{finitary} when all of its circuits are finite.
Since restrictions and contractions of a finitary matroid are again finitary, the following is clear from \Cref{theorem: main theorem}.

\begin{cor}\label{finitary proto-exact}
  The full subcategory $\MSF$ of $\MS$ with objects the finitary matroids is proto-exact with the induced proto-exact structure.
\end{cor}

The following is the core result of this section, and drives the remainder of our results.

\begin{prop}\label{finitize matroid}
  The collection $\fin{\mc{C}}$ of finite circuits of matroid $M$ is also a set of circuits of a matroid $\fin{M}$.
  Moreover, the identity map on $E(M)$ is a strong map $\fin{M} \to M$.
\end{prop}

\begin{proof}
  The circuit elimination axiom yields that eliminations between finite circuits are again finite, and the directed union of every chain of $\fin{\mc{C}}$-independent sets is trivially $\fin{\mc{C}}$-independent; thus $\fin{\mc{C}}$ is a set of matroid circuits.
  We compute as follows to verify condition \ref{closure containment} of \Cref{strong map conditions}.
  \begin{align*}
    \closure{fin}(S)
    & =
      S \cup \set{e \in E(M)}{\exists C \in \circuits{M} \text{ finite} \text{ with } e \in C \subseteq S \cup e}
    \\
    & \subseteq
      S \cup \set{e \in E(M)}{\exists C \in \circuits{M} \text{ with } e \in C \subseteq S \cup e}
    \\
    & =
      \closure{M}(S)
      .
    & \qedhere
  \end{align*}
\end{proof}

\begin{prop}\label{subset system}
  The matroid $\fin{M}$ is the co-limit of the diagram $\set{M | S \hookrightarrow M | T}{ S \subseteq T \text{ finite}}$ in $\MS$.
\end{prop}

\begin{proof}
  We prove that $\fin{M}$ satisfies the universal property of co-limits.
  Let $N$ be a matroid and $M | S \xrightarrow{\alpha_S} N$ be a strong map for each $S \subseteq E(M)$ such that $\alpha_T|_S = \alpha_S$ for all pairs $S \subseteq T$ of finite sets.
  Applying the forgetful functor $\FF \colon \MS \to \SET$, we obtain a unique co-limit map $E(M) \xrightarrow{\alpha} E(N)$ of pointed sets.
  We claim that $\alpha$ is a strong map $\fin{M} \xrightarrow{\alpha} N$.
  Let $A \subseteq E(M)$ and $e \in \fin{\closure{M}}(A) \setminus A$ be arbitrary.
  There is a finite circuit $C$ of $M$ with $e \in C \subseteq A \cup e$ by \Cref{closure by circuits}.
  Hence
  \[ %
    \alpha (e) %
    \in \alpha (C) %
    \subseteq \alpha (\closure{M | C}(C) )%
    \subseteq \closure{N}(\alpha (C)) %
  \] %
  yields that $\alpha$ is a strong map by \Cref{strong map conditions}.
\end{proof}

\begin{cor}\label{proposition: finitary matroids}
  A matroid is finitary if and only if it is a co-limit of finite matroids.
\end{cor}

\begin{proof}
  If $M$ is finitary, then $M = \fin{M}$ is a co-limit of finite matroids by \Cref{subset system}.
  Conversely, suppose that
  \[
    M = \colim \set{M_i \xrightarrow{f_{i, j}} M_j}{i, j \in I}
  \]
  is a co-limit of finite matroids with strong maps $M_i \xrightarrow{g_i} M$.
  Note that $\closure{}(A) = \closure{fin}(A)$ for all finite $A \subseteq E(M)$ by \Cref{closure by circuits}.
  Thus the set maps $g_i$ are also strong maps for matroids $M_i \to \fin{M}$.
  This induces a strong map $M \to \fin{M}$, which is necessarily the identity as a set map.
  Hence $M = \fin{M}$ by \Cref{finitize matroid}.
\end{proof}

We obtain the following as a special case of \Cref{proposition: finitary matroids}.

\begin{cor}
  A matroid on a countable set is finitary if and only if it is a direct limit of finite matroids.
\end{cor}

\begin{proof}
  If $M$ is finitary on a countable ground set, then identifying the elements of $M$ with $\N$ it is easy to show that $M$ is the co-limit of $M|\{0, 1, ..., n\} \hookrightarrow M|\{0, 1, ..., n+1\}$ where $0$ is represents the distinguished loop.
  The reverse implication follows trivially from \Cref{proposition: finitary matroids} and the fact that a direct limit of finite sets is countable.
\end{proof}

\begin{cor}\label{cor: induce fin}
  Every strong map $M \xrightarrow{f} N$ induces a strong map $\fin{M} \xrightarrow{f^{fin}} \fin{N}$.
  In particular, $\fin{\bullet} \colon \MS \to \MSF$ constitutes a full functor.
\end{cor}

\begin{proof}
  The map $f$ yields strong maps $M | S \to N | f S \to \fin{N}$ for all $S \subseteq E(M)$.
  Hence we obtain a strong map $\fin{M} \to \fin{N}$ by \Cref{subset system}.
  The second statement is clear from \Cref{finitize matroid}; indeed $\fin{\bullet}$ is a left inverse of the inclusion functor $\MSF \hookrightarrow \MS$.
\end{proof}

\subsection{Finite Presentability}

An object $X$ in a category $\mc{C}$ is \emph{finitely presented} when the functor $\Hom_{\mc{C}}(X, -)$ preserves filtered co-limits.
We now investigate the finitely presented matroids.

The functor $\GG \colon \SET \to \MS$ sending a pointed set $E$ to the free (pointed) matroid on $E$ has a left adjoint, namely the forgetful functor $\FF \colon \MS \to \SET$.
In particular, there is a natural bijection between the sets $\Hom_{\MS}(\GG E, N)$ and $\Hom_{\SET}(E, \FF N)$; indeed, the natural isomorphism is the identity map.
On the other hand, a set $X$ is finite if and only if $\Hom(X, -)$ preserves filtered co-limits, so $\Hom(\GG E, -)$ preserves filtered co-limits if and only if $E$ is finite.
Hence the free matroid on $E$ is finitely presented if and only if $E$ is finite.
In fact, this argument generalizes as follows.

\begin{prop}\label{proposition: finitely presentable}
  A matroid is finitely presented if and only if it is finite.
\end{prop}

\begin{proof}
  Let $M$ be a matroid.
  Assuming $M$ is finitely presented, we note by \Cref{finitize matroid} that morphisms $M \to N$ induce $\fin{M} \to \fin{N}$.
  As $\fin{M}$ is a co-limit of its finite restrictions by \Cref{subset system}, we have $\colim_S \Hom(M, M | S) \cong \Hom(M, \colim_S M | S)$.
  Thus there is an element of the co-limit $f = \colim_S (f_S \colon M \to M_S)$ and $i_S \colon M | S \to M$ satisfying $i_S \circ f = \id_{\fin{M}}$.
  As a set map, this is the identity; thus $i_S$ is surjective, yielding $M$ is finite.

  Conversely, assuming $M$ is finite and $\colim_i N_i$ is filtered, then note $f \colon M \to \colim_i N_i$ has image contained in the ground set of $N_j$ for some $j \in I$; thus $f \in \Hom(M, N_j)$.
  As $\colim \Hom(M, N_i) = \bigcup_i \Hom(M, N_i)$ we have that the identity function is the desired bijection $\colim_i \Hom(M, N_i) \cong \Hom(M, \colim N_i)$.
\end{proof}

In particular, not all finitary matroids are finitely presented in $\MS$ (or even in $\MSF$).

\begin{rmk}
  We initially hoped that the category of finitary matroids might be locally finitely presentable, but this fails spectacularly.
  The full subcategory of finite matroids is neither complete nor co-complete as it has neither products nor co-equalizers by \cite[Propositions 3.5 and 3.7]{heunen2018category}.
\end{rmk}

%%%%%%%%%%%%%%%%%%%%%%%%%%%%%%%%%%%%%%%%%%%%%%%%%%%%%%%%%%%%%%%%%%%%%%%%
\section{Future Directions}

Our initial motivation for this project was to employ categorical language to study infinite matroids, generalizing theorems from finite matroid theory when possible.
We now discuss some future directions.

\subsection{Tutte-Grothendieck Invariants}

Our initial main goal was to use categorical language to introduce a Tutte-Grothendieck invariant for infinite matroids.
We discuss here an impossibility result we obtained on Tutte-Grothendieck invariants for infinite matroids.\footnote{The example described below was also discovered independently by R.\ Pendavingh (private communication).}

Ideally such an invariant should be determined by homomorphisms from some ring, as in the finite case with the general framework of \cite{brylawski1972tgring}.
It should satisfy (at least) the following conditions to account for the matroid structure:
\begin{enumerate}
\item \label{TG generators} %
  There is a generator for each matroid isomorphism class.
\item \label{TG relations} %
  For all matroids $M$ and all elements $e$ of $M$ which is neither a loop nor a coloop we have $[M] = [M \setminus e] + [M / e]$.
\end{enumerate}
Now letting $E$ be any infinite set, let $U_r(E)$ denote the finitary uniform matroid of rank $r \geq 1$ on $E$ (i.e.\ the circuits of $U_r(E)$ are the subsets of $E$ of size $r+1$), and let $e \in E$.
Now we compute
\[
  [U_{r}(E)] %
  = [U_{r}(E) \setminus e] + [U_{r}(E) / e] %
  = [U_{r}(E \setminus e)] + [U_{r-1}(E \setminus e)] %
  = [U_{r}(E)] + [U_{r-1}(E)] %
  .
\]
Hence $U_{r}(E) = 0$ for all $r \geq 0$ and all infinite sets $E$.
Hence \emph{every ring} satisfying both (\ref{TG generators}) and (\ref{TG relations}) above must identify the finitary uniform matroids on infinite sets with $0$.
In particular, no \emph{ring} can be a universal isomorphism invariant of infinite matroids which respects single-element deletion-contraction relations and evaluates nontrivially on all matroids.

One way to circumvent this issue is to give up additive inverses, instead seeking a universal object satisfying (\ref{TG generators}) and (\ref{TG relations}) in the spirit of \emph{semirings}.

\begin{que}
  Develop a meaningful generalization of the Tutte-Grothendieck theory of finite matroids for infinite matroids using the ideas above.
  Can this theory be practically applied to solve other problems on infinite matroids?
\end{que}

\subsection{$K$-Theory for Infinite Matroids}

The categories $\MS$ and $\MSF$ are proto-exact by \Cref{theorem: main theorem} and \Cref{finitary proto-exact}, so we have a $K$-theory using the Waldhausen construction.
This suggests the following problem:

\begin{que}
  Compute the $K$-theory of $\MS$ and/or $\MSF$.
\end{que}

We directly computed the Grothendieck group of the category of finite matroids in \cite[Theorem 6.3]{eppolito2018proto}.
Our computation relies heavily on finitarity (in the category-theoretic sense) of the category of finite matroids; we showed that the class of a finite $M$ can be identified with $(r, \#E(M) - r)$ where $r$ is the rank of $M$.
Neither $\MS$ nor $\MSF$ is a finitary category, and the following \emph{a priori} simpler problem seems like the place to start.

\begin{que}
  Compute the class of $U_r(E)$ in $K_0(\MSF)$ for all $r \geq 0$ and all sets $E$.
\end{que}

%%%%%%%%%%%%%%%%%%%%%%%%%%%%%%%%%%%%%%%%%%%%%%%%%%%%%%%%%%%%%%%%%%%%%%%%
\printbibliography
\end{document}